\documentclass{article}
\usepackage[utf8]{inputenc}
\usepackage{amsthm}

\newtheorem{thm}{Theorem}
\newtheorem{lemma}{Lemma}
\newtheorem*{defi}{Definition} \theoremstyle{remark}
 \theoremstyle{plain}
\newtheorem{cor}[thm]{Corollary} 

\title{Subdivision rules for all Gromov hyperbolic groups}
\author{Brian Rushton}
\begin{document}
\maketitle

\abstract{This paper shows that every Gromov hyperbolic group can be described by a finite subdivision rule acting on the 3-sphere. This gives a boundary-like sequence of increasingly refined finite cell complexes which carry all quasi-isometry information about the group. This extends a result from Cannon and Swenson in 1998 that hyperbolic groups can be described by a recursive sequence of overlapping coverings by possibly wild sets, and demonstrates the existence of non-cubulated groups that can be represented by subdivision rules.}

\section{Introduction}

\begin{defi}
A geodesic metric space satisfies the \textbf{thin triangles condition} if there is a global constant  $\delta > 0$ such that each edge of a geodesic triangle is contained within the $\delta$-neighborhood of the other two edges. Such a geodesic metric space is also called $\delta$-\textbf{hyperbolic}.

A \textbf{(Gromov) hyperbolic group} is a group with a Cayley graph that satisfies the thin triangles condition for some $\delta>0$.
\end{defi}

Recursive structures on hyperbolic groups have played an important role in geometric and combinatorial group theory. This research area started with Cannon's paper on linear recursions for hyperbolic groups \cite{cannon1984combinatorial}. Inspired by this, Thurston developed the concept of an automatic group, which was expanded on by Epstein and others \cite{epstein1992word}. Subdivision rules, which are the topic of this paper, were first described in \cite{hyperbolic}, where Cannon and Swenson showed that the boundary at infinity of a hyperbolic group can be described by a recursively defined sequence of possibly wild compact sets with large overlap (see Theorem 3.30 of the above paper).

In one special case worked out by Cannon, Floyd, and Parry, these coverings were actually tilings of the 2-sphere by polygons, a particularly nice type of covering. They defined finite subdivision rules as a topological space and map that recursively generate a sequence of tilings on a sphere \cite{subdivision}, in order to model coverings similar to those generated by negative sets. As they said in \cite{subdivision}, "Finite subdivision rules model the recursive structures of sequences of disk covers arising from negatively curved groups. The above sequences of disk covers do not generally arise from finite subdivision rules mainly because pairs of distinct elements of a disk cover may have large intersection, whereas finite subdivision rules deal with tilings, in which distinct elements can only intersect in their boundaries." Thus, finite subdivision rules were used to model the behavior of hyperbolic groups at infinity, but only one group was known to be described by a finite subdivision rule at infinity.

The definition of a subdivision rule was expanded by the author in \cite{CubeSubs} and \cite{rushton2014classification} to allow finite subdivision rules in other dimensions. In this paper, we show that all hyperbolic groups can be represented by a finite subdivision rule in dimension 3.

We give the expanded definition here:

\begin{defi}

A map between cell complexes is \textbf{cellular} if it is a homeomorphism when restricted to each open cell.

A cell complex is \textbf{almost polyhedral} if the gluing maps in the cell complex are cellular, and if the subcells of each closed cell are contained in the cell's boundary.

A \textbf{finite subdivision rule} is an almost polyhedral cell complex $S$ equipped with a cellular map $\phi_S$ (called the \textbf{subdivision map}) onto itself. A finite subdivision rule $S$ \textbf{acts} on a cell complex $X$ if there is a cellular map $f$ (called the \textbf{structure map}) from $S$ to $X$. If $S$ acts on $X$, then $X$ is called an $S$-complex. Then the $n$\textbf{-th subdivision} $S^n(X)$ of $X$ is the new cell structure on $X$ obtained by pulling back the map $\phi_S^n\circ f:X\rightarrow S$.

\end{defi}

The cell structures $X=S^0(X), S^1(X), S^2(X),...$ on $X$ are nested. Thus, a finite subdivision rule creates a sequence of coverings of $X$ by compact sets (specifically, cells with disjoint interiors).

\begin{defi}
Let $(X,d_X), (Y,d_Y)$ be metric spaces. A map $f:X\rightarrow Y$ is called a \textbf{quasi-isometry} if there is a constant $K$ such that:
\begin{enumerate}
\item for any points $x,y$ in $X$, $\frac{1}{K}d_X(x,y)-K\leq d_Y(f(x),f(y))\leq Kd_X(x,y)+K$
\item every point of $Y$ is within $K$ of the image of $X$
\end{enumerate}
\end{defi}

In this paper, we show that, for any hyperbolic group, there is a finite subdivision rule acting on the 3-sphere from which the original group can be reconstructed, up to quasi-isometry. We say that the subdivision rule \textbf{represents} the hyperbolic group.

As mentioned earlier, finite subdivision rules were initially studied as a means of understanding the recursive structure of a hyperbolic group's boundary at infinity. However, for many years, only one example of a hyperbolic group with a finite subdivision rule on the boundary at infinity was known \cite{subdivision}. The author expanded the list of known subdivision rules to hyperbolic 3-manifolds created from right-angled polyhedra \cite{PolySubs}. This paper greatly enlarges the list of groups known to be described by finite subdivision rules. In particular, this theorem provides the first examples of non-cubulated groups that can be represented by finite subdivision rules.

In \cite{rushton2014classification}, we show that groups with subdivision rules have an atlas that translates quasi-isometry properties of the group into combinatorial properties of the subdivision rule. Thus, all of the correspondences in that atlas apply to these subdivision rules for hyperbolic groups.

\section{Lexicographical ordering and cone types}

\begin{defi}Given an ordering $\prec$ on a symmetric generating set for a group $G$, the geodesics in the Cayley graph of $G$ can be ordered lexicographically; that is, two geodesics are ordered according to the first generator by which they differ. This is a total order \cite{calegari2013ergodic}. Then, given multiple geodesics from the origin to a vertex $x$, there is a unique geodesic $u$ such that $u \prec v$ for all geodesics $v$ from the origin to $x$. Such a geodesic is called \textbf{lexicographically first} among all geodesics from the origin to $v$ \cite{calegari2013ergodic}.

A \textbf{(finitely) labeled graph} is a graph together with a map from the edges of the graph to a finite set of \textbf{edge labels}, and a map from the vertices of the graph to a finite set of \textbf{vertex labels}. For purposes of this article, we include unions of open edges as labeled graphs.

For background on labeled graphs, see \cite{gallian2009dynamic}.

A \textbf{deterministic finite state automaton} is a finite directed labeled graph  with the following requirements. The vertex labels are unique, and are called \textbf{states}, while the edge labels are called \textbf{inputs}. One state is chosen as the \textbf{initial state}, and a subset of states is chosen as \textbf{accept states}. Finally, the set of inputs coming out of each vertex is independent of vertices, i.e. it is the same for all vertices.

A \textbf{regular language} is a subset of the free monoid $M$ on the set of inputs of a deterministic finite state automaton. In particular, it consists of exactly the words which describe a path in the deterministic finite state automaton from an initial state to an accept state.

A regular language is \textbf{prefix-closed} if, whenever $t=su$ in $M$ with $t$ in the regular language, $s$ is also in the regular language.

\end{defi}

\begin{thm}\label{PreTheorem}
(Theorem 3.2.2 of \cite{calegari2013ergodic} and its proof) The set of all words in generators of $G$ that represent lexicographically first geodesics is a prefix-closed regular language with inputs consisting of single generators.
\end{thm}

\begin{cor}
The union of all lexicographically first geodesics is a tree.
\end{cor}

\begin{proof}
The prefix of one lexicographically first geodesic is another lexicographically first geodesic, by Theorem \ref{PreTheorem}. Thus, any given point in Cayley graph of $G$ has a unique path to the origin. By moving each point along its unique geodesic to the origin, we can see that the Cayley graph is contractible, and thus a tree.
\end{proof}

\begin{defi}
A \textbf{labeled graph morphism} is a graph morphism between labeled graphs that preserves labels.
\end{defi}

\begin{defi}
The \textbf{open star} of a vertex is the union of a vertex with all the open edges that have that vertex as an endpoint.
\end{defi}

\begin{defi}Given $g \in G$ the cone type of $g$, denoted $cone(g)$, is the set of $h \in G$ for which
some geodesic from $id$ to $gh$ passes through $g$.

For any $n$, the $n$-level of $g$ is the
set of $h$ in the ball $B_n(id)$ such that $|gh| < |g|$. \end{defi}

We reproduce the following theorem:

\begin{thm}\label{CannonTheorem}
(Lemma 7.1 of \cite{cannon1984combinatorial}). The $2\delta + 1$ level of an element determines its cone type.
\end{thm}

\begin{cor}\label{CannonCor}
Let $\Gamma=\Gamma(G)$ be the union of all lexicographically first geodesics in the Cayley graph of $G$. Then $\Gamma$ has finitely many cone types.
\end{cor}

\begin{proof}
By Theorem \ref{PreTheorem}, the words representing lexicographically first geodesics are a regular language; thus, two words that end in the same accept state have the same set of words that can be appended to them, and thus have the same cone type. Therefore, there are only finitely many cone types.
\end{proof}









\section{History graphs and combinatorial subdivision rules}

\begin{defi}Given a subdivision rule $S$ with subdivision map $\phi_S$, we can arbitrarily choose a subset $I$ of $S$ which is the union of closed cells and which maps into itself under $\phi_S$, and call it the \textbf{ideal set} of $S$. Its complement is called the $\textbf{limit set}$, and is denoted by $L$. The \textbf{ideal set and limit set} of an $S$-complex $X$ are the subsets of $X$ which map onto $I$ and $L$, respectively, under the map $f$. The ideal set and limit set of $S^n(X)$ are denoted $I_n$ and $L_n$. A subdivision rule with a choice of ideal and limit set is called a \textbf{colored subdivision rule}.

For purposes of this paper, we define the \textbf{dual graph} of $L_n$ to be the graph consisting of one vertex for each cell of $L_n$, with edges corresponding to inclusion.

Given a finite subdivision rule $S$ acting on a cell complex $X$, the \textbf{history graph} $\Gamma=\Gamma(S,X)$ is the graph consisting of the union of:

\begin{enumerate}
\item A single vertex $O$ called the \textbf{origin}, together with
\item Disjoint subgraphs $\Gamma_n$, each of which is isomorphic to the dual graph of $L_n$, and
\item Edges between $\Gamma_n$ and $\Gamma_{n+1}$ corresponding to subdivision; i.e. an edge connecting a vertex of $\Gamma_n$ (corresponding to a cell $C$) to all vertices of $\Gamma_{n+1}$ whose corresponding cells are contained in $C$, and
\item Edges connecting the origin $O$ to each vertex of $\Gamma_0$.
\end{enumerate}
\end{defi}

\begin{defi}We say that a subdivision rule $S$ acting on a complex $X$ \textbf{represents a metric space} $(Y,d_Y)$, if the history graph $\Gamma(S,X)$ with the path metric is quasi-isometric to $Y$. We say that it \textbf{represents a group} $G$ if $\Gamma(S,X)$ is quasi-isometric to the Cayley graph of G, using path metrics.
\end{defi}

The process of creating a history graph can be reversed; in \cite{rushton2015all}, we defined a \textbf{combinatorial subdivision graph} as a graph $\Xi$ which contains a family of disjoint subgraphs $\Xi_n$ such that:
\begin{enumerate}
\item $\Xi_0$ is a single vertex.
\item Every vertex is contained in some $\Xi_n$.
\item Every vertex $v$ of $\Xi_n$ for $n>0$ is connected to a unique vertex of $\Xi_{n-1}$ called the \textbf{predecessor} of $v$. We define the predecessor of the unique vertex in $\Xi_0$ to be itself.
\item If two vertices of $\Xi_n$ are connected by an edge for some $n>0$, then their predecessors are connected by an edge or are the same vertex.
\item The open stars of any two vertices with the same label are labeled-graph isomorphic.
\item Conditions 3 and 4 allow us to define a graph morphism $\pi:\Xi\rightarrow\Xi$ which sends each vertex to its predecessor. We call the map $\pi$ the \textbf{predecessor map}. Then we require the preimages under $\pi$ of two edges with the same label to be labeled-graph isomorphic. A representative graph in such an isomorphism class is called an \textbf{edge subdivision}. Similarly, we require the preimage of two open stars of vertices with the same label to be labeled-graph isomorphic, and a representative graph in this isomorphism class is called a \textbf{vertex subdivision}.
\end{enumerate}

We then have the following theorem \cite{rushton2015all}:

\begin{thm}Given a combinatorial subdivision graph $\Xi$, there is a finite subdivision rule $S$ acting on a complex $X$ such that the history graph $\Gamma(S,X)$ is quasi-isometric to $\Xi$.
\end{thm}

In the following section, given a hyperbolic group $G$, we construct a combinatorial subdivision graph quasi-isometric to the Cayley graph of $G$, and use the above theorem to represent $G$ by a finite subdivision rule $S$ and complex $X$.

\section{Proof of main theorem}

\begin{defi}
Given a hyperbolic group $G$, we let $\Gamma=\Gamma(G)$ be the union of all lexicographically first geodesics in the Cayley graph of $G$.

Given $v$ in $\Gamma$, we let $i(v)$ be the corresponding vertex in the Cayley graph of $G$.

We say that two vertices $u_1,u_2$ of $\Gamma$ are \textbf{geodesically close} if if $u_1,u_2$ are the same distance from the origin, and if there are geodesics $\gamma_i$ in the Cayley graph of $G$ starting from $i(u_i)$ in $G$ that pass within a distance of 1 from each other at some pair of vertices which are no closer to the origin than $u_1$ and $u_2$.

We let $\Xi=\Xi(G)$ be the graph that contains $\Gamma$ and also contains an edge connecting each pair of geodesically close vertices.

Given $v$ in $\Xi$, we let $j(v)$ be the corresponding vertex in the Cayley graph of $G$.

When we speak of 'the corresponding vertex' in discussing these 3 graphs, it is implied that we are using these bijections.

\end{defi}

\begin{lemma}\label{CloseLem}
If two vertices $u_1,u_2$ of $Gamma$ are geodesically close, then  $i(u_1)$ and $i(u_2)$ are within $2\delta +1$ of each other in the Cayley graph of $G$. 
\end{lemma}

\begin{proof}
Let $i^{-1}(u_1),i^{-1}(u_2)$ be two vertices in $\Gamma$ that are geodesically close (so that $u_1,u_2$ are in the Cayley graph of $G$). Then there are vertices $u_1',u_2'$ in the Cayley graph of $G$ that are each contained in geodesics from the origin through $u_1,u_2$ respectively, such that $u_1'$ and $u_2'$ are connected by an edge. Let $x$ be the midpoint of this edge. Then for $i=1,2$ there is a geodesic $\gamma_i$ from the origin to $x$ passing through $u_i$ and $u'_i$. Consider the degenerate geodesic triangle consisting of the two geodesics $\gamma_1$ and $\gamma_2$. By the thin triangles condition, $u_1$ must be within $\delta$ of some point $v$ in $\gamma_2$. By the definition of geodesically close and by the triangle inequality, $|d(O,v)-d(O,u_2)|=|d(O,v)-d(O,u_1)|\leq d(u_1,v)\leq \delta$. Then because $v$ and $u_2$ are on the same geodesic, and their distances from the origin are within $\delta$ of each other, then $v$ and $u_2$ are within $\delta$ of each other. Again by the triangle inequality, $d(u_1,u_2)\leq d(u_1,v)+d(v,u_2)\leq \delta+\delta < 2\delta+1$.
\end{proof}

\begin{defi}

A \textbf{cone} $K$\textbf{neighborhood} of an element $g$ in a group $G$ consists of:

\begin{enumerate}
\item the set of all elements $h$ with $|h|<K$ such that $i^{-1}(g)$ and $i^{-1}(gh)$ are geodesically close, and
\item the natural map from the above set to the set of cone types of the Cayley graph of $G$.
\end{enumerate}
\end{defi}

\begin{thm}
Every Gromov hyperbolic group is quasi-isometric to the history graph of a (3-dimensional) finite subdivision rule.
\end{thm}

\begin{proof}
Let $\Xi_n$ be the sphere of radius $n$ about the identity in $\Xi$, and let $\Xi_0$ be the identity.

Then conditions 1 and 2 of the definition of a combinatorial subdivision graph are satisfied automatically by $\Xi$ and its subgraphs $\Xi_n$. Because there is a unique lexicographically first geodesic to each vertex in $\Gamma$ (and thus in $\Xi$), condition 3 is also satisfied.

To see that condition 4 is satisfied, let $e$ be an edge between two vertices $u_1,u_2$ of $\Xi_n$. Then the corresponding vertices in $\Gamma$ are geodesically close, so there are geodesics $\gamma_i$ starting from $i(u_i)$ that approach within a distance of 1. However, by appending the segments connecting each $i(u_i)$ to its predecessor, we see that the predecessors of the $i(u_i)$ are either equal or geodesically close. Thus the $u_i$ are either equal or connected by an edge, satisfying condition 4.

By Lemma \ref{CloseLem}, if two vertices $j(u_1),j(u_2)$ of the Cayley graph of $G$ have the same $2\delta+1$-cone neighborhoods, then the corresponding vertices $u_1,u_2$ in $\Xi$ have the same set of horizontal edges (i.e. edges contained in a single $\Xi_n$) extending from them. They will also have the same vertical edges, which are determined by the cone type of the vertex itself. Thus, the corresponding vertices $u_1,u_2$ in $\Xi$ will have isomorphic open stars (with horizontal edges being determined by the neighborhood and vertical edges being determined by the level; recall that every vertex besides the identity has exactly one vertical, upwards-leading edge). 

We now show that the vertex subdivision of a vertex  $x$ in $\Xi$ is determined by the cone $2\delta +1$ neighborhood of $j(x)$ in the Cayley graph of $G$. First, consider $\Gamma$. By Corollary \ref{CannonCor}, $\Gamma$ has only finitely many cone types, and the vertical edges of $\Gamma$ are the vertical edges of $\Xi$; by definition, two vertices $x, x'$ in $\Xi$ with the same cone type will have the isomorphic sets (with the isomorphism preserving cone types) of \emph{vertices} $x_1,x_2,...,x_k$ and $x'_1,x'_2,...,x'_k$ in the preimage under the predecessor map. Next, the cone $2\delta+1$ neighborhoods of $x$ and of $x'$ determine the cone $2\delta +1$ neighborhoods of the vertices $x_1,x_2,...$ and $x'_1,x'_2,....$ respectively, which by the preceding paragraph shows that if $x_i$ and $x_j$ are connected by an edge, then $x'_i$ and $x'_j$ are connected by an edge). This is not yet a labeled isomorphism, as we have not defined the labels for edges.

We now do so. We label all vertical edges with the same label. Then, given a horizontal edge $e$ in $\Xi$ with vertices $x$ and $y$, we label it by the cone types of $j(x)$ and $j(y)$, and by the \textbf{relative group element} $j(x)^{-1}j(y)$.

Consider such a horizontal edge $e$ in $\Xi$ with vertices $x$ and $y$. If $f$ is in the preimage of $e$ under the predecessor map, then it has a vertex $u$ in the preimage of $x$ and a vertex $v$ in the preimage of $y$, and $u$ and $v$ must be geodesically close. Conversely, given two vertices that are geodesically close, with one in the preimage of $x$ and the other in the preimage of $y$, the edge containing them will be in the preimage of $e$. Thus, the edges in the preimage are determined completely by the vertices in the preimages of the endpoints and by the relative distances between them. The former are determined by the cone types, and the latter by both the cone types and the relative group element $xy^{-1}$ (which allows us to know the relative position of the two groups of vertices). Also, the cone types of two vertices determine the cone types of their preimages under the predecessor map, and the cone type plus the relative group element also allows us to determine the relative group elements $u^{-1}v$ for all preimages $u$ and $v$. Thus, there are finitely many edge types.

We can now define vertex labels: we label each vertex's open star by its cone $2\delta+1$ neighborhood. This determines all of the edge types of the open star, as well as the cone $2\delta+1$-neighborhoods of the vertices in the preimage (and thus their labels). Thus, $\Xi$ is a combinatorial subdivision rule.


We need to show that $\Xi$ is quasi-isomorphic to the Cayley graph of $G$. It is sufficient to show that if 2 vertices are connected by an edge in one graph then the corresponding vertices are less than $2\delta+2$ apart in the other graph, and vice versa.

Let $u,v$ be vertices in $\Xi$ connected by a vertical edge. Then since the vertical edges all correspond to lexicographically first edges in the Cayley graph of $G$, the corresponding vertices in the Cayley graph of $G$ are connected by an edge.

Let $u,v$ be vertices in $\Xi$ connected by a horizontal edge. Then the corresponding vertices are geodesically close, so by Lemma \ref{CloseLem}, they are less than $2\delta+1<2\delta+2$ apart.

Now let $u,v$ be vertices in the Cayley graph of $G$ connected by a horizontal edge. Because they are only 1 apart, they are geodesically close, and their corresponding vertices in $\Xi$ are connected by an edge as well.

If $u,v$ in the Cayley graph of $G$ are connected by a vertical edge, with $u$ further from the origin than $v$, then let $u'$ be the next closest element to the origin along the lexicographically first geodesic to $u$. Then there are geodesics in the Cayley graph of $G$ passing through $u',v$ and leading to the same vertex $u$, so the vertices corresponding to $u'$ and $v$ are connected by a horizontal edge in $\Xi$, and $u$ and $u'$ are connected by a vertical edge in $\Xi$, for a total distance of 2 or less, which is less than $2\delta +2$.




\end{proof}

\bibliographystyle{plain} \bibliography{hypersubs}

\begin{thebibliography}{10}

\bibitem{calegari2013ergodic}
Danny Calegari.
\newblock The ergodic theory of hyperbolic groups.
\newblock {\em Contemp. Math}, 597:15--52, 2013.

\bibitem{subdivision}
J.~W. Cannon, W.~J. Floyd, and W.~R. Parry.
\newblock Finite subdivision rules.
\newblock {\em Conformal Geometry and Dynamics}, 5:153--196, 2001.

\bibitem{hyperbolic}
J.~W. Cannon and E.~L. Swenson.
\newblock Recognizing constant curvature discrete groups in dimension 3.
\newblock {\em Transactions of the American Mathematical Society},
  350(2):809--849, 1998.

\bibitem{cannon1984combinatorial}
James~W Cannon.
\newblock The combinatorial structure of cocompact discrete hyperbolic groups.
\newblock {\em Geometriae Dedicata}, 16(2):123--148, 1984.

\bibitem{epstein1992word}
David Epstein, Mike~S Paterson, James~W Cannon, Derek~F Holt, Silvio~V Levy,
  and William~P Thurston.
\newblock {\em Word processing in groups}.
\newblock AK Peters, Ltd., 1992.

\bibitem{gallian2009dynamic}
Joseph~A Gallian.
\newblock A dynamic survey of graph labeling.
\newblock {\em The electronic journal of combinatorics}, 16(6):1--219, 2009.

\bibitem{PolySubs}
B.~Rushton.
\newblock Constructing subdivision rules from polyhedra with identifications.
\newblock {\em Alg. and Geom. Top.}, 12:1961--1992, 2012.

\bibitem{CubeSubs}
B.~Rushton.
\newblock A finite subdivision rule for the n-dimensional torus.
\newblock {\em Geometriae Dedicata}, pages 1--12, 2012.

\bibitem{rushton2014classification}
Brian Rushton.
\newblock Classification of subdivision rules for geometric groups of low
  dimension.
\newblock {\em Conformal Geometry and Dynamics of the American Mathematical
  Society}, 18(10):171--191, 2014.

\bibitem{rushton2015all}
Brian Rushton.
\newblock All finite subdivision rules are combinatorially equivalent to
  three-dimensional subdivision rules.
\newblock {\em arXiv preprint arXiv:1512.00367}, 2015.

\end{thebibliography}

\end{document}